\documentclass[11pt]{article}
\usepackage[letterpaper,margin=0.98in]{geometry}
\usepackage[T1]{fontenc}
\usepackage{lmodern}
\usepackage{amsmath,amssymb,amsthm}
\usepackage{microtype}
\usepackage{needspace}
\usepackage{xcolor}
\usepackage{hyperref}
\definecolor{citationblue}{RGB}{0,65,200}
\hypersetup{
  colorlinks=true,
  linkcolor=citationblue,
  citecolor=citationblue,
  urlcolor=citationblue,
  filecolor=citationblue,
  pdfborder={0 0 0},
  pdftitle={BMO Spaces on Chord-Arc Curves},
  pdfauthor={Tailiang Liu and Yuliang Shen}}
\newcommand{\R}{\mathbb R}
\newcommand{\C}{\mathbb C}
\newcommand{\Chat}{\overline{\mathbb C}}
\newcommand{\Hup}{\mathbb H}
\newcommand{\Hlow}{\mathbb L}
\newcommand{\D}{\mathbb D}
\newcommand{\T}{\mathbb T}
\newcommand{\BMO}{\mathrm{BMO}}
\newcommand{\BMOA}{\mathrm{BMOA}}
\newcommand{\SQS}{\mathrm{SQS}}
\newcommand{\ellg}{\ell}
\newtheorem{theorem}{Theorem}

\title{Three BMO-type Spaces Characterizing Chord--Arc Curves}
\author{Tailiang Liu\thanks{School of Mathematics and Physics, Jiangsu University of Technology, Changzhou 213001, China. Email: \href{mailto:ltlmath@jsut.edu.cn}{\texttt{ltlmath@jsut.edu.cn}}.}\qquad
Yuliang Shen\thanks{Department of Mathematics, Soochow University, Suzhou 215006, China. Email: \href{mailto:ylshen@suda.edu.cn}{\texttt{ylshen@suda.edu.cn}}.}}
\date{}
\begin{document}
\maketitle
\begin{abstract}
In this note, we introduce three BMO-type spaces on a locally rectifiable Jordan curve and prove that it is a chord-arc curve if and only if these spaces coincide with equivalent norms.
\end{abstract}

\section{The main result and its proof}
Let $\C$ be the complex plane, $\Chat=\C\cup\{\infty\}$, $\R$ the real line, and $\Hup$ and $\Hlow$ the upper and lower half-planes, respectively. Recall that a locally integrable function $f$ on $\R$ belongs to $\BMO(\R)$ if
\begin{equation}\label{eq:bmo-R}
 \|f\|_{\BMO(\R)}
 =\sup_{I\subset\R}\frac{1}{|I|}\int_I|f(x)-f_I|\,dx<\infty,
 \qquad f_I=\frac{1}{|I|}\int_I f(x)\,dx,
\end{equation}
where $I$ ranges over bounded intervals and $|I|$ denotes their length.

Let $\Gamma\subset\Chat$ be a locally rectifiable Jordan curve passing through $\infty$, and let $ds=|d\zeta|$ be arc-length measure on $\Gamma$. Analogously, $\BMO(\Gamma)$ consists of all locally integrable functions $f$ for which
\begin{equation}\label{eq:bmo-Gamma}
 \|f\|_{\BMO(\Gamma)}
 =\sup_{J\subset\Gamma}
 \frac{1}{\ellg(J)}\int_J|f(\zeta)-f_J|\,ds<\infty,
 \qquad f_J=\frac{1}{\ellg(J)}\int_Jf\,ds,
\end{equation}
where  $\ellg(J)$ denotes  length of $J$. An arc-length parametrization $\gamma:\R\to\Gamma$ evidently induces an isometric isomorphism $f\mapsto f\circ\gamma$ between $\BMO(\Gamma)$ and $\BMO(\R)$.

Recall also the M\"obius invariance of BMO (see Garnett~\cite{Garnett}). If $\tau:\Hup\to\D$ is a M\"obius transformation, where $\D$ is the unit disk and $\T=\partial\D$, then
\begin{equation}\label{eq:mob}
 \|f\circ\tau\|_{\BMO(\R)}\asymp\|f\|_{\BMO(\T)},
 \qquad f\in\BMO(\T).
\end{equation}
Here  $A\asymp B$ means that $A$ and $B$ are comparable up to positive constants. This motivates conformally invariant BMO spaces on $\Gamma$.

Let $\Omega^+$ and $\Omega^-$ be the complementary domains of $\Gamma$, and choose conformal maps $\varphi:\Hup\to\Omega^+$ and $\psi:\Hlow\to\Omega^-$ fixing $\infty$. Their boundary extensions are homeomorphisms from $\R\cup\{\infty\}$ onto $\Gamma$. Since $\Gamma$ is locally rectifiable, these extensions are locally absolutely continuous on $\R$ and their boundary derivatives are nonzero almost everywhere (see~\cite{Duren}). Hence, for a measurable  function $f$ with respect to $ds$, the pullbacks below are well defined up to Lebesgue-null sets. Set
\begin{equation}\label{eq:bmo-pm}
 \begin{aligned}
 \BMO^+(\Gamma)&=\{f:f\circ\varphi\in\BMO(\R)\},
 &\quad\|f\|_{\BMO^+(\Gamma)}&=\|f\circ\varphi\|_{\BMO(\R)},\\
 \BMO^-(\Gamma)&=\{f:f\circ\psi\in\BMO(\R)\},
 &\quad\|f\|_{\BMO^-(\Gamma)}&=\|f\circ\psi\|_{\BMO(\R)}.
 \end{aligned}
\end{equation}
 The definitions of $\BMO^\pm(\Gamma)$ do not depend on the choices of $\varphi,\psi$, up to equivalent seminorms (indeed, any two such normalized maps differ by an affine automorphism of the half-plane).

The comparison of these spaces is suggested by the authors' study of the BMO boundedness of the Cauchy integral on chord-arc curves~\cite[Section~3]{LiuShen}. Recall that $\Gamma$ is \emph{chord-arc} if, for some $C>0$,
$\ellg(\Gamma[z,w])\le C|z-w|$ for all finite $z,w\in\Gamma$, where $\Gamma[z,w]$ is the subarc between $z$ and $w$ not containing $\infty$. We obtain the following characterization.

\begin{theorem}\label{thm:main}
A locally rectifiable Jordan curve $\Gamma\subset\Chat$ passing through $\infty$ is chord-arc if and only if
\begin{equation}\label{eq:theorem}
 \BMO(\Gamma)=\BMO^+(\Gamma)=\BMO^-(\Gamma)
\end{equation}
as spaces of boundary functions, with equivalent seminorms.
\end{theorem}

\begin{proof}
Recall that an increasing homeomorphism $h:\R\to\R$ is \emph{strongly quasisymmetric}, denoted by $h\in\SQS(\R)$, if it is locally absolutely continuous and $h'$ is a Muckenhoupt $A_\infty$ weight. Jones' theorem~\cite{Jones} states that $h\in\SQS(\R)$ if and only if the composition operator $b\mapsto b\circ h$ is bounded on $\BMO(\R)$; in this case the operator is a bounded isomorphism.

By choosing the orientation of the arc--length parametrization $\gamma$ appropriately, we obtain two increasing homeomorphisms $h_+,h_-\colon\mathbb R\to\mathbb R$ such that $\gamma\circ h_+=\varphi$ and $\gamma\circ h_-=\psi$.
Thus $h_+'=|\varphi'|$ and $h_-'=|\psi'|$ almost everywhere.

If $\Gamma$ is chord-arc, Lavrentiev's theorem~\cite{Lavrentiev} yields $|\varphi'|,|\psi'|\in A_\infty(\R)$, so $h_\pm\in\SQS(\R)$.
For any measurable function $f$ on $\Gamma$, we have
\[
f\circ\varphi = (f\circ\gamma)\circ h_+,\qquad
f\circ\psi = (f\circ\gamma)\circ h_-.
\]
Hence $f\circ\varphi\in \mathrm{BMO}(\mathbb R)$ if and only if $f\circ\gamma\in \mathrm{BMO}(\mathbb R)$, with equivalent seminorms; the same holds for $f\circ\psi$.
Therefore $\mathrm{BMO}(\Gamma)=\mathrm{BMO}^+(\Gamma)=\mathrm{BMO}^-(\Gamma)$ with equivalent seminorms.
 This is precisely the argument used in ~\cite[proof of Theorem~3.2]{LiuShen}.

Conversely, suppose~\eqref{eq:theorem} holds. For any $b\in\BMO(\R)$, we have
\[
 \|b\circ h_\pm\|_{\BMO(\R)}
 \asymp\|b\circ\gamma^{-1}\|_{\BMO^\pm(\Gamma)}
 \asymp\|b\|_{\BMO(\R)}.
\]
By Jones' theorem, $h_\pm\in\SQS(\R)$; hence
\begin{equation}\label{eq:Ainfty}
 |\varphi'|,|\psi'|\in A_\infty(\R).
\end{equation}
Moreover, the conformal welding $h_\Gamma=\psi^{-1}\circ\varphi=h_-^{-1}\circ h_+$ is strongly quasisymmetric (the class $\SQS(\R)$ is a group). In particular, $h_\Gamma$ is quasisymmetric, so $\Gamma$ is a quasicircle. \Needspace{6\baselineskip}The BMO welding theorem of Astala and Zinsmeister~\cite[Theorem 4]{AstalaZinsmeister}, applied to $h_\Gamma$ and its inverse, gives
\begin{equation}\label{eq:AZ}
 \log\varphi'\in\BMOA(\Hup),\qquad
 \log\psi'\in\BMOA(\Hlow),
\end{equation}
where $\BMOA$ denotes the space of analytic BMO functions. As BMOA functions are the Poisson extensions of their boundary values, taking real parts in~\eqref{eq:AZ} gives
\[
 \log|\varphi'(x+iy)|
 =\frac1\pi\int_\R\frac{y\log|\varphi'(t)|}{(x-t)^2+y^2}\,dt,
 \qquad y>0,
\]
and the analogous formula for $\psi$ in $\Hlow$. Thus both $\Omega^+$ and $\Omega^-$ are Smirnov domains (see \cite{Pommerenke}). Together with~\eqref{eq:Ainfty} and the quasicircle property, the converse to Lavrentiev's theorem due to Jerison and Kenig~\cite[Theorem~4.2]{JerisonKenig} implies that $\Gamma$ is chord-arc.
\end{proof}


\begin{thebibliography}{99}
\footnotesize
\setlength{\itemsep}{0pt}
\setlength{\parsep}{0pt}     
\setlength{\lineskip}{0pt}

\bibitem{AstalaZinsmeister}
K.~Astala and M.~Zinsmeister, \href{https://doi.org/10.1007/BF01446592}{Teichm\"uller spaces and BMOA}, \emph{Math. Ann.} \textbf{289} (1991), 613--625.
\bibitem{Duren}
P. L. Duren, \emph{Theory of $H^p$ spaces}, volume 38 of Pure and Applied Mathematics. Academic Press, New York--London, 1970.
\bibitem{Garnett}
J.~B. Garnett, \emph{Bounded Analytic Functions}, Academic Press, New York, 1981.
\bibitem{JerisonKenig}
D.~S. Jerison and C.~E. Kenig, \href{https://doi.org/10.7146/math.scand.a-11956}{Hardy spaces, $A_\infty$, and singular integrals on chord-arc domains}, \emph{Math. Scand.} \textbf{50} (1982), 221--247.
\bibitem{Jones}
P.~W. Jones, \href{https://doi.org/10.1007/BF02384312}{Homeomorphisms of the line which preserve BMO}, \emph{Ark. Mat.} \textbf{21} (1983), 229--231.
\bibitem{Lavrentiev}
M.~A. Lavrentiev, Boundary problems in the theory of univalent functions, \emph{Amer. Math. Soc. Transl. (2)} \textbf{32} (1963), 1--35.
\bibitem{LiuShen}
T.~Liu and Y.~Shen, \href{https://doi.org/10.1007/s10114-024-2184-4}{The Faber operator acting on BMOA, BMO-Teichm\"uller space and chord-arc curves}, \emph{Acta Math. Sin. (Engl. Ser.)} \textbf{40} (2024), 2359--2387.
\bibitem{Pommerenke}
Ch.~Pommerenke, \href{https://doi.org/10.1007/978-3-662-02770-7}{\emph{Boundary Behaviour of Conformal Maps}}, Springer-Verlag, Berlin, 1992.
\end{thebibliography}
\end{document}